\documentclass[11pt,a4paper]{article}

\usepackage{inputenc}
\usepackage{amsmath}
\usepackage{bm}
\usepackage{bbold}
\usepackage{amsthm}
\usepackage{enumerate}

\usepackage[hyphens]{url}

\usepackage{hyperref}
\usepackage{breakurl}

\usepackage{tikz}\tikzset{x=1cm,y=1cm,z=1cm}
\usetikzlibrary{perspective}

\usepackage{pgfplots}\pgfplotsset{compat=1.16}

\title{Algebraic solution to constrained bi-criteria decision problem of rating alternatives through pairwise comparisons
\thanks{Mathematics, 2021, 9(4), 303, doi:10.3390/math9040303}}

\author{N. Krivulin\thanks{Faculty of Mathematics and Mechanics, Saint Petersburg State University, 28 Universitetsky Ave., St.~Petersburg, 198504, Russia, nkk@math.spbu.ru.}
\thanks{This work was supported in part by the Russian Foundation for Basic Research (grant No. 20-010-00145).}}

\date{}

\newtheorem{theorem}{Theorem}
\newtheorem{lemma}[theorem]{lemma}
\newtheorem{corollary}[theorem]{corollary}

\theoremstyle{definition}
\newtheorem{example}{Example}

\begin{document}

\maketitle

\begin{abstract}
We consider a decision-making problem to evaluate absolute ratings of alternatives from the results of their pairwise comparisons according to two criteria, subject to constraints on the ratings. We formulate the problem as a bi-objective optimization problem of constrained matrix approximation in the Chebyshev sense in logarithmic scale. The problem is to approximate the pairwise comparison matrices for each criterion simultaneously by a common consistent matrix of unit rank, which determines the vector of ratings. We represent and solve the optimization problem in the framework of tropical (idempotent) algebra, which deals with the theory and applications of idempotent semirings and semifields. The solution involves the introduction of two parameters that represent the minimum values of approximation error for each matrix and thereby describe the Pareto frontier for the bi-objective problem. The optimization problem then reduces to a parametrized vector inequality. The necessary and sufficient conditions for solutions of the inequality serve to derive the Pareto frontier for the problem. All solutions of the inequality, which correspond to the Pareto frontier, are taken as a complete Pareto-optimal solution to the problem. We apply these results to the decision problem of interest and present illustrative examples.
\\

\\

\end{abstract}

\section{Introduction}

This paper is concerned with an application of tropical algebra to a bi-criteria decision problem of rating alternatives by pairwise comparisons. Tropical (idempotent) algebra deals with the theory and applications of algebraic systems with idempotent operations, typically defined as taking the maximum and minimum of two arguments. Since the first publications in 1950--60s, models and methods of tropical mathematics have found increased use in solving various problems in operations research, computer science and other fields. Rapid advances in the area are demonstrated in many published works, including the recent monographs and textbooks \cite{Baccelli1993Synchronization,Kolokoltsov1997Idempotent,Golan2003Semirings,Heidergott2006Maxplus,Mceneaney2006Maxplus,Gondran2008Graphs,Maclagan2015Introduction}. 

One of the current research directions in tropical mathematics are optimization problems which can be defined and solved in terms of tropical algebra. These problems are often formulated as the minimization or maximization of functions on idempotent semifields (algebraic systems with idempotent addition and invertible multiplication). For some notable applications, including time constrained project scheduling, minimax location problems with Chebyshev and rectilinear distances, and decision making through pairwise comparisons, tropical optimization can provide a direct complete solution that explicitly describes all solutions in a compact parametric form, ready for formal analysis and instant computations with a decent polynomial time complexity. 

The problem of evaluating absolute ratings (priorities, scores, weights) of alternatives (choices, decisions, possibilities) from the results of their pairwise comparisons according to several criteria is of great practical importance in multi-criteria decision making \cite{Saaty1990Analytic,Gavalec2015Decision}. The most commonly used approaches to handle this problem in terms of conventional mathematics are the Analytical Hierarchy Process (AHP) method, developed by T.~L.~Saaty~\cite{Saaty1990Analytic,Saaty2013Onthemeasurement}, and the Weighted Geometric Means (WGM) method \cite{Crawford1985Note,Barzilai1997Deriving}. The AHP method is based on a heuristic procedure that provides a numerical solution whose accuracy is normally accepted by practitioners, but optimality cannot be theoretically guaranteed. The WGM method offers an analytical result in a rather simple form, which proves to be a Pareto-optimal solution and thus is formally justified as optimal. Both methods, however, can hardly be used or extended to obtain all Pareto-optimal solutions of the pairwise comparison problem, which are of particular interest in multi-criteria optimization.    

The existing literature on the multi-criteria pairwise comparison problem, which presents a variety of solution methods and application examples, mainly focuses on unconstrained problems where no specific restrictions are imposed on the values of ratings under evaluation. At the same time, increasing complexity of contemporary decision-making processes, including decision models, methods and operating conditions, calls for the development of new decision approaches to take into account possible conditions on the ratings, such as order relations or box constraints fixed in advance. 

In the framework of tropical algebra, the problem of rating alternatives by pairwise comparisons is examined in a range of papers, including \cite{Elsner2004Maxalgebra,Elsner2010Maxalgebra,Gursoy2013Analytic,Tran2013Pairwise,Goto2020Polyad}. A new approach is proposed in \cite{Krivulin2015Rating,Krivulin2016Using,Krivulin2020Using,Krivulin2019Tropical}, which offers a direct complete solution in an explicit parametric form. Specifically, a tropical analogue of the AHP is developed in \cite{Krivulin2016Using,Krivulin2019Tropical}, and a constrained single-criterion problem is solved in \cite{Krivulin2015Rating}. For an unconstrained bi-criteria problem, a complete Pareto-optimal solution is obtained in \cite{Krivulin2020Using}. 

In this paper, we consider a new decision-making problem of rating alternatives through pairwise comparisons according to two criteria, subject to constraints on the ratings. In practice, the pairwise comparison data normally come from human judgments of experts (analysts, specialists, referees) or non-experts (customers, consumers, users), whereas the constraints result from conditions and limitations, which are inherent in the nature of alternatives or the decision-making process.

We follow the general solution scheme proposed in \cite{Krivulin2020Using} to solve the bi-criteria problem without constraints. This scheme is further extended below to develop a new general solution that accommodates constraints in an efficient way. We start with the formulation of the problem as a bi-objective optimization problem of constrained matrix approximation in the Chebyshev sense in logarithmic scale ($\log$-Chebyshev approximation). The problem is to approximate the pairwise comparison matrices for each criterion simultaneously by a common consistent matrix of unit rank, which determines the vector of ratings.

Furthermore, we represent and solve the optimization problem in the tropical algebra setting. The solution approach involves the introduction of two parameters that represent the minimum values of the approximation error for each matrix and thereby describe the Pareto frontier for the bi-objective problem. The optimization problem then reduces to a parametrized vector inequality. The necessary and sufficient conditions for solutions of this inequality serve to derive the Pareto frontier for the optimization problem. All solutions of the inequality, which correspond to the Pareto frontier, are taken as a complete Pareto-optimal solution of the optimization problem.

The rest of the paper proceeds as follows. In Section~\ref{S-CBCDP}, we describe and discuss the bi-criteria decision-making problem of interest, which motivates the study. \mbox{Section~\ref{S-PADR}} offers a brief overview of basic facts about tropical algebra, which are used in the subsequent solutions. (This section can be skipped by readers familiar with the subject and notation.) Section~\ref{S-CBOOP} provides a complete Pareto-optimal solution to a constrained bi-objective tropical optimization problem in an exact analytical form. We apply the result obtained to the decision-making problem under consideration and present illustrative examples in \mbox{Section~\ref{S-ACBCDP}}. Section~\ref{S-C} includes some concluding remarks.

\section{Constrained Bi-Criteria Decision Problem}
\label{S-CBCDP}

The purpose of this section is to describe and discuss the decision-making problem which serves to motivate and illustrate the study. Suppose that one needs to evaluate alternatives in a decision-making process of selecting alternatives according to their ratings. Given relative results of pairwise comparison of alternatives on a continuous scale, obtained with respect to two criteria, the problem is to derive absolute ratings of alternatives, subject to constraints imposed on the ratings. The problem of rating alternatives dates back to the classical work by L.~L.~Thurstone \cite{Thurstone1927Law} in the first part of the last century, and since that time has been the subject of numerous investigations.

\subsection{Unconstrained Pairwise Comparisons Under Single Criterion}

Assume that $n$ alternatives are compared in pairs under a single criteria, which results in a pairwise comparison matrix $\bm{A}=(a_{ij})$, where the entry $a_{ij}>0$ indicates that alternative $i$ is $a_{ij}$ times superior (more preferred) than $j$. The entries of $\bm{A}$ satisfy the equality $a_{ij}=1/a_{ji}$, and thus this matrix is positive and symmetrically reciprocal.  

A pairwise comparison matrix $\bm{A}$ is called consistent if its entries possess the transitivity property $a_{ij}=a_{ik}a_{kj}$, which corresponds to the natural transitivity of judgments. If the matrix $\bm{A}$ is consistent, then there exists a unique (up to a positive factor) positive vector $\bm{x}=(x_{i})$, which determines the entries of $\bm{A}$ by the conditions $a_{ij}=x_{i}/x_{j}$. It follows from these conditions that the entries of $\bm{x}$ directly specify the individual ratings of alternatives and hence completely solve the problem of interest.

However, the pairwise comparison matrices, which are encountered in real-world problems, are commonly not consistent due to various cognitive or technical limitations of the comparison process. To overcome this difficulty, several techniques \cite{Saaty1984Comparison,Saaty1990Analytic,Choo2004Common}, which range from various heuristic procedures to matrix approximation, are used to replace a given inconsistent matrix $\bm{A}$ of pairwise comparisons by a consistent matrix $\bm{X}=(x_{i}/x_{j})$ that is close to $\bm{A}$ in some sense. Since the matrix $\bm{X}$ is completely determined by a vector $\bm{x}=(x_{i})$ of individual ratings, these techniques typically concentrate on the direct derivation of the vector $\bm{x}$ rather than of the matrix $\bm{X}$.

The heuristic solutions are often based on different schemes of aggregating columns in the pairwise comparison matrix, such as the method of weighted column sums \cite{Choo2004Common}. A commonly used technique of deriving weights from pairwise comparisons is the principal eigenvector method \cite{Saaty1977Scaling,Saaty1990Analytic,Saaty2013Onthemeasurement}. The method exploits the principal (Perron) eigenvector of the pairwise comparison matrix as a vector of ratings. Despite the wide applications of the principal eigenvector method and many convincing arguments presented in favor of this method, it cannot guarantee, in a mathematically strict sense, that the solution obtained is~optimal. 

In contrast to the heuristic methods, the techniques that derive approximating matrices by minimizing a distance between matrices (approximation error) offer a commonly accepted and mathematically justified approach to the problem. The available procedures differ according to distance functions and measurement scales used to evaluate the approximation error. The application of Euclidean (Frobenius), rectilinear (Manhattan) and Chebyshev metrics on the standard linear scale in approximation of pairwise comparison matrices normally results in complicated multiextremal nonlinear optimization problems that are hard to solve \cite{Saaty1984Comparison,Choo2004Common}, and thus has no wide use.

The approximation in logarithmic scale with logarithm to a base greater than one leads to optimization problems that can usually be solved numerically by an appropriate computational algorithm, and, in some cases, analytically \cite{Saaty1984Comparison,Barzilai1997Deriving,Choo2004Common,Portugal2011Weberfechner}. Minimization of the Euclidean distance in logarithmic scale ($\log$-Euclidean approximation) offers a unique (up to a positive factor) solution, which is given parametrically in an exact analytical form. Due to the simplicity of solution and the rigorous formal justification, the $\log$-Euclidean approximation finds extensive application in rating alternatives from pairwise comparisons, where it is known as the method of geometric means.   

Both rectilinear and Chebyshev approximation in logarithmic scale can be reduced by an appropriate transformation to solving linear programs using one of the computational algorithms available in linear programming. This algorithmic approach, however, cannot provide the derivation of a complete analytical result that allows describing all solutions in a direct explicit form.  

The implementation of $\log$-Chebyshev approximation involves minimizing the maximum absolute differences between the logarithms of the corresponding entries of a pairwise comparison matrix $\bm{A}=(a_{ij})$ and consistent matrix $\bm{X}=(x_{i}/x_{j})$ by solving the problem to
\begin{equation}
\begin{aligned}
&
\text{minimize}
&&
\max_{1\leq i,j\leq n}
\left|\log a_{ij}-\log\frac{x_{i}}{x_{j}}\right|.
\end{aligned}
\label{P-min_maxijlogaijlogxixj}
\end{equation}

Note that this problem can be reduced to an equivalent problem in the same variables, which does not involve logarithms (see, e.g., \cite{Krivulin2020Using,Krivulin2019Tropical}). Indeed, representing the absolute value function as the maximum of two opposite values, using the monotonicity of the logarithm and the condition $a_{ij}=1/a_{ji}$ turn the objective function into
\begin{equation*}
\max_{1\leq i,j\leq n}
\left|\log a_{ij}-\log\frac{x_{i}}{x_{j}}\right|
=
\log\max_{1\leq i,j\leq n}\max\left\{\frac{a_{ij}x_{j}}{x_{i}},\frac{x_{i}}{a_{ij}x_{j}}\right\}
=
\log\max_{1\leq i,j\leq n}\frac{a_{ij}x_{j}}{x_{i}}.
\end{equation*}

Moreover, the monotonicity property allows replacing the minimization of the logarithm by minimizing its argument, and therefore, the problem of $\log$-Chebyshev approximation of $\bm{A}=(a_{ij})$ at \eqref{P-min_maxijlogaijlogxixj} reduces to finding positive vectors $\bm{x}=(x_{i})$ that
\begin{equation}
\begin{aligned}
&
\text{minimize}
&&
\max_{1\leq i,j\leq n}\frac{a_{ij}x_{j}}{x_{i}}.
\end{aligned}
\label{P-min_maxijaijxixj}
\end{equation}

A solution approach to handle problem \eqref{P-min_maxijaijxixj} is proposed in \cite{Krivulin2015Rating,Krivulin2016Using,Krivulin2019Tropical} in the framework of tropical algebra. Using this approach, a complete analytical solution is derived, which describes all solution vectors in a compact parametric form.

\subsection{Constrained Pairwise Comparisons Under Two Criteria}

Suppose now that $n$ alternatives are compared in pairs according to two (unweighted) criteria to produce two pairwise comparison matrices $\bm{A}=(a_{ij})$ and $\bm{B}=(b_{ij})$. The problem of rating alternatives takes the form of finding a common consistent matrix $\bm{X}=(x_{i}/x_{j})$ that is close to (or approximates) both matrices $\bm{A}$ and $\bm{B}$ simultaneously. The new problem has two objectives that, in general, are in conflict, and thus requires the application of multi-objective optimization techniques.   

A wide accepted approach to solve the multi-criteria problems of pairwise comparisons uses the AHP decision method \cite{Saaty1990Analytic,Saaty2013Onthemeasurement}, which is based on the principal eigenvector calculation. In the case of the bi-criteria problem in question, the method produces a unique numerical solution in the form of the sum of normalized principal eigenvectors of the matrices $\bm{A}$ and $\bm{B}$ (taken with equal weights). Another approach follows the WGM method \cite{Crawford1985Note,Barzilai1997Deriving}, which offers a unique analytical result, where the elements of the vector of ratings are calculated by multiplying the corresponding row geometric means of pairwise comparison matrices and taking square roots of the results.    

A solution in terms of tropical algebra that, which applies the $\log$-Chebyshev approximation and provides a direct analytical representation of the result, is developed in \cite{Krivulin2015Rating,Krivulin2016Using,Krivulin2019Tropical}. The solution uses a scalarization technique to reduce the vector-valued bi-criteria problem to a single criterion problem in the form of \eqref{P-min_maxijaijxixj} with the matrix $\bm{A}$ replaced by the matrix $\bm{D}=(d_{ij})$, where $d_{ij}=\max\{a_{ij},b_{ij}\}$. 

Since no single solution generally exists to satisfy all objectives simultaneously, the solutions, which yield the best compromise between objectives, are normally of primary interest in multi-objective problems. A common approach to obtain the best compromising solution is to derive a set of Pareto-optimal (Pareto-efficient, non-dominated) solutions at which none of the objectives can be improved without making another objective worse~\cite{Ehrgott2005Multicriteria,Luc2008Pareto,Benson2009Multiobjective,Pappalardo2008Multiobjective}. In bi-objective problems, the image of the Pareto-optimal set, referred to as the Pareto frontier, can often be visualized as a trade-off curve in the plane of objectives, and used to describe all Pareto-optimal solutions \cite{Ruzika2005Approximation}.

In the framework of $\log$-Chebyshev approximation, the bi-criteria problem of pairwise comparisons takes the form of the bi-objective optimization problem
\begin{equation*}
\begin{aligned}
&
\text{minimize}
&&
\left(
\max_{1\leq i,j\leq n}\frac{a_{ij}x_{j}}{x_{i}},
\max_{1\leq i,j\leq n}\frac{b_{ij}x_{j}}{x_{i}}
\right),
\end{aligned}
\end{equation*}
which has a complete Pareto-optimal solution given in the tropical algebra setting in \cite{Krivulin2020Using}, where the solution set is explicitly described in a parametric vector form. 

Let us now assume that the absolute ratings of alternatives must satisfy additional constraints, which restrict possible relations between individual ratings. These constraints may utilize prior knowledge or reflect approved and adopted results of previous studies, which are independent of the current comparison data. Specifically, an alternative can be known a priori to be preferable to another by virtue of all aspects of comparison, regardless of subjective judgments.

As an example, consider the problem of reconstruction of a total order on a set of alternatives, where the information on relations between some elements is unknown (lost, hidden, deteriorated) to make the set partially ordered. A reasonable way to restore the total order is to estimate the ratings of alternatives from pairwise comparisons (provided, for instance, by one or more experts), while preserving the known relations in the form of constraints on the ratings. The order is then defined by the ranks of alternatives, deduced from ratings obtained by solving this constrained problem.

Another illustration is a two-stage procedure of evaluating alternatives, which is to combine the outcome of both stages in such a way that some results of the first stage override the results of the second. Suppose that the first stage separates the set of alternatives into groups according to an important property. For example, the set can be divided into the groups of relatively low, medium and high estimated costs incurred by taking the alternatives. The group affiliation is considered as the dominant criterion and can be represented by constraints, which involves that the ratings of alternatives of the first group are not less than those of the second, and the ratings of the second are not less than the third. At the second stage, the ratings of all alternatives are evaluated using pairwise comparisons by experts according to criteria other than and independent of cost levels (the information about the costs may be confidential and hidden from the experts). The final ratings are derived from the results of pairwise comparisons under the constraints imposed by group affiliation.

Given real numbers $c_{ij}\geq0$, the constraints can be represented by the inequalities
\begin{equation*}
c_{ij}x_{j}
\leq
x_{i},
\qquad
i,j=1,\ldots,n,
\end{equation*}
where $c_{ij}$ specifies that the rating of alternative $i$ must be not less than $c_{ij}$ times the rating of $j$. The value $c_{ij}=1$ indicates that the rating of $i$ is not less than $j$, whereas $c_{ij}=0$ shows that no lower bound is defined on the rating of $i$ with respect to $j$. Combining the inequalities for all $j$ into one yields
\begin{equation*}
\max_{1\leq j\leq n}
c_{ij}x_{j}
\leq
x_{i},
\qquad
i=1,\ldots,n.
\end{equation*}

The bi-objective problem now turns into the next problem: given symmetrically reciprocal matrices $\bm{A}=(a_{ij})$ and $\bm{B}=(b_{ij})$, and non-negative matrix $\bm{C}=(c_{ij})$, find a positive vector $\bm{x}=(x_{i})$ to
\begin{equation}
\begin{aligned}
&
\text{minimize}
&&
\left(
\max_{1\leq i,j\leq n}\frac{a_{ij}x_{j}}{x_{i}},
\max_{1\leq i,j\leq n}\frac{b_{ij}x_{j}}{x_{i}}
\right);
\\
&
\text{subject to}
&&
\max_{1\leq j\leq n}
c_{ij}x_{j}
\leq
x_{i},
\quad
i=1,\ldots,n.
\end{aligned}
\label{P-min_maxijaijxixj-maxijbijxixj-maxijxjleqxi}
\end{equation}

In the subsequent sections, a complete Pareto-optimal solution of problem \eqref{P-min_maxijaijxixj-maxijbijxixj-maxijxjleqxi} is derived in a direct parametric form in terms of tropical algebra.

\section{Preliminary Algebraic Definitions and Results}
\label{S-PADR}

We start with a brief overview of main definitions, basic facts and preliminary results of tropical algebra, based mainly on \cite{Krivulin2015Extremal,Krivulin2015Multidimensional,Krivulin2017Direct}, to provide a formal framework to the solution of the bi-objective optimization problem in what follows. For further details on tropical mathematics and its applications, one can consult, for example, the monographs and textbooks~\cite{Baccelli1993Synchronization,Kolokoltsov1997Idempotent,Golan2003Semirings,Heidergott2006Maxplus,Mceneaney2006Maxplus,Gondran2008Graphs,Maclagan2015Introduction}.

\subsection{Idempotent Semifields}

Consider a set $\mathbb{X}$ that is closed under operations $\oplus$ (addition) and $\otimes$ (multiplication), and includes their neutral elements $\mathbb{0}$ (zero) and $\mathbb{1}$ (one). An algebraic structure $(\mathbb{X},\oplus,\otimes,\mathbb{0},\mathbb{1})$ is called an idempotent semifield if $(\mathbb{X},\oplus,\mathbb{0})$ is a commutative idempotent monoid (semilattice), $(\mathbb{X}\setminus\{\mathbb{0}\},\otimes,\mathbb{1})$ is an Abelian group, and multiplication $\otimes$ distributes over addition $\oplus$.

For each $x\ne\mathbb{0}$ in the semifield, its multiplicative inverse is denoted by $x^{-1}$. The power notation with integer exponents specifies iterated products defined for each $x\ne\mathbb{0}$ and integer $p>0$ as $\mathbb{0}^{p}=\mathbb{0}$, $x^{0}=\mathbb{1}$, $x^{p}=xx^{p-1}$, $x^{-p}=(x^{-1})^{p}$. (Here and hereafter the multiplication sign $\otimes$ is, as usual, omitted for the sake of brevity.) The equation $x^{p}=a$ is assumed to have a unique solution $x$ for any $a\in\mathbb{X}$ and integer $p>0$, which extends the power notation to rational exponents.

Idempotent addition introduces a partial order on $\mathbb{X}$ by the rule: $x\leq y$ if and only if $x\oplus y=y$. With this order, both addition and multiplication are monotone in each arguments, which means that the inequality $x\leq y$ yields the inequalities $x\oplus z\leq y\oplus z$ and $xz\leq yz$ for $x,y,z\in\mathbb{X}$. Exponentiation is monotone: the inequality $x\leq y$ results in $x^{r}\geq y^{r}$ if $r<0$, and $x^{r}\leq y^{r}$ if $r\geq0$ for $x,y\ne\mathbb{0}$. Addition possesses the extremal property (the majority law) that $x\leq x\oplus y$ and $y\leq x\oplus y$. Finally, the inequality $x\oplus y\leq z$ is equivalent to the pair of inequalities $x\leq z$ and $y\leq z$. 

In what follows, the above partial order is assumed extended to a linear order to make the semifields under consideration totally ordered. 

A typical example of the idempotent semifield is the system $\mathbb{R}_{\max}=(\mathbb{R}_{+},\max,\times,0,1)$, where $\mathbb{R}_{+}$ is the set of non-negative reals. This semifield, which is often called the max-algebra, has addition defined as maximum, and multiplication as usual. The zero $\mathbb{0}$ and one $\mathbb{1}$ respectively coincide with the arithmetic $0$ and $1$, the power notation and the notion of inversion have the ordinary interpretation, and the order induced by idempotent addition corresponds to the natural linear order on $\mathbb{R}_{+}$.

\subsection{Matrices and Vectors}

Matrices over $\mathbb{X}$ are introduced in the ordinary way. The set of matrices with $m$ rows and $n$ columns is denoted $\mathbb{X}^{m\times n}$. A matrix with all entries equal to $\mathbb{0}$ is the zero matrix denoted by $\bm{0}$. A matrix without zero columns is called column-regular.

Matrix operations follow the conventional rules with the arithmetic addition and multiplication replaced by $\oplus$ and $\otimes$. The scalar inequalities which represent properties of scalar operations extend to the matrix operations, where the inequalities are interpreted~entry-wise.  

Consider square matrices of order $n$ from $\mathbb{X}^{n\times n}$. A matrix with all diagonal entries equal to $\mathbb{1}$ and the non-diagonal entries to $\mathbb{0}$ is the identity matrix $\bm{I}$. 

The power notation with non-negative integer exponents specifies repeated multiplication of a matrix by itself, and is defined as $\bm{A}^{0}=\bm{I}$ and $\bm{A}^{p}=\bm{A}\bm{A}^{p-1}$ for any matrix $\bm{A}$ and integer $p>0$.

The trace of a matrix $\bm{A}=(a_{ij})$ is given by $\mathop\mathrm{tr}\bm{A}=a_{11}\oplus\cdots\oplus a_{nn}$ and possesses usual properties, including its invariance under cyclic permutations of matrices.

For any matrix $\bm{A}\in\mathbb{X}^{n\times n}$, a tropical analogue of matrix determinant is a function given by
\begin{equation*}
\mathop\mathrm{Tr}(\bm{A})
=
\mathop\mathrm{tr}\bm{A}
\oplus\cdots\oplus
\mathop\mathrm{tr}\bm{A}^{n}.
\end{equation*}

Provided that $\mathop\mathrm{Tr}(\bm{A})\leq\mathbb{1}$, the asterate operator (Kleene star) is defined as
\begin{equation*}
\bm{A}^{\ast}
=
\bm{I}\oplus\bm{A}
\oplus\cdots\oplus
\bm{A}^{n-1}.
\end{equation*}

A matrix with one column (row) is a column (row) vector. The set of column vectors of order $n$ is denoted by $\mathbb{X}^{n}$. All vectors below are column vectors if not explicitly transposed. A vector with all elements equal to $\mathbb{0}$ is the zero vector denoted $\bm{0}$. A vector without zero elements is called regular.

Conjugate transposition of a nonzero vector $\bm{x}=(x_{j})$ yields the row vector $\bm{x}^{-}=(x_{j}^{-})$, where $x_{j}^{-}=x_{j}^{-1}$ if $x_{j}\ne\mathbb{0}$, and $x_{j}^{-}=\mathbb{0}$ otherwise.
A vector $\bm{y}$ is collinear with $\bm{x}$ if $\bm{y}=c\bm{x}$ for some $c\in\mathbb{X}$.

A scalar $\lambda\in\mathbb{X}$ is an eigenvalue of a matrix $\bm{A}\in\mathbb{X}^{n\times n}$ if there exists a nonzero vector $\bm{x}\in\mathbb{X}^{n}$, called an eigenvector of $\bm{A}$, such that $\bm{A}\bm{x}=\lambda\bm{x}$. The maximum eigenvalue (with respect to the order induced by idempotent addition) is called the spectral radius of the matrix and calculated as
\begin{equation*}
\lambda
=
\mathop\mathrm{tr}\bm{A}
\oplus\cdots\oplus
\mathop\mathrm{tr}\nolimits^{1/n}(\bm{A}^{n}).
\end{equation*}

The zero and identity matrices over the max-algebra $\mathbb{R}_{\max}$ have the same form as in conventional linear algebra. The matrix and vector operations are performed in $\mathbb{R}_{\max}$ by the standard rules with the arithmetic addition replaced by the maximum operation. The regular vectors are positive vectors.

\subsection{Vector Inequalities}

We now describe preliminary results that play a key role in the solution of the bi-objective tropical optimization problem in the next section. First assume that, given a matrix $\bm{A}\in\mathbb{X}^{m\times n}$ and a vector $\bm{d}\in\mathbb{X}^{m}$, we need to solve, with respect to the unknown vector $\bm{x}\in\mathbb{X}^{n}$, the inequality
\begin{equation}
\bm{A}\bm{x}
\leq
\bm{d}.
\label{I-Axleqd}
\end{equation}

This inequality has a well-known solution that is available in various forms (see, e.g.,~\cite{Baccelli1993Synchronization}). In what follows, we use the solution provided by the following statement \cite{Krivulin2015Extremal}.
\begin{lemma}
\label{L-Axleqd}
For any column-regular matrix $\bm{A}$ and regular vector $\bm{d}$, all solutions of inequality \eqref{I-Axleqd} are given by the inequality $\bm{x}\leq(\bm{d}^{-}\bm{A})^{-}$.
\end{lemma}

Next suppose that, for a given square matrix $\bm{A}\in\mathbb{X}^{n\times n}$, we seek to find regular vectors $\bm{x}\in\mathbb{X}^{n}$ to satisfy the inequality
\begin{equation}
\bm{A}\bm{x}
\leq
\bm{x}.
\label{I-Axleqx}
\end{equation}

To solve the problem, we apply the next result, which is obtained in \cite{Krivulin2015Extremal,Krivulin2015Multidimensional} and offers a complete solution of the inequality in a parametric form.
\begin{theorem}
\label{T-Axleqx}
For any square matrix $\bm{A}$, the following statements hold.
\begin{enumerate}
\item
If $\mathop\mathrm{Tr}(\bm{A})\leq\mathbb{1}$, then all regular solutions to \eqref{I-Axleqx} are given by $\bm{x}=\bm{A}^{\ast}\bm{u}$, where $\bm{u}$ is any regular vector.
\item
If $\mathop\mathrm{Tr}(\bm{A})>\mathbb{1}$, then there is only the trivial solution $\bm{x}=\bm{0}$.
\end{enumerate}
\end{theorem}

Note that this theorem describes all regular solutions, if exist, as the set of regular vectors (the linear span) generated by the columns in the Kleene star matrix $\bm{A}^{\ast}$.

\subsection{Identities for Traces}

We conclude the overview with binomial identities for matrices and their traces, which allow to simplify subsequent algebraic manipulations. We start with an identity that is valid for any matrices $\bm{A},\bm{B}\in\mathbb{X}^{n\times n}$ and integer $m>0$ in the following form (see also \cite{Krivulin2017Direct}):
\begin{equation*}
(\bm{A}\oplus\bm{B})^{m}
=
\bm{A}^{m}
\oplus
\bigoplus_{k=1}^{m}\bigoplus_{\substack{i_{0}+i_{1}+\cdots+i_{k}=m-k\\i_{0},i_{1},\ldots,i_{k}\geq0}}
\bm{A}^{i_{0}}(\bm{B}\bm{A}^{i_{1}}\cdots\bm{B}\bm{A}^{i_{k}}).
\end{equation*}

Taking the trace of both sides, using the permutation-invariant property of traces to change the order of matrix factors, and renaming the indices yield
\begin{equation*}
\mathop\mathrm{tr}(\bm{A}\oplus\bm{B})^{m}
=
\mathop\mathrm{tr}\bm{A}^{m}
\oplus
\bigoplus_{k=1}^{m}\bigoplus_{\substack{i_{1}+\cdots+i_{k}=m-k\\i_{1},\ldots,i_{k}\geq0}}
\mathop\mathrm{tr}(\bm{A}^{i_{1}}\bm{B}\cdots\bm{A}^{i_{k}}\bm{B}).
\end{equation*} 

After summing over $m=1,\ldots,n$, and rearranging terms, we have the identity
\begin{equation}
\mathop\mathrm{Tr}(\bm{A}\oplus\bm{B})
=
\mathop\mathrm{Tr}(\bm{A})
\oplus
\bigoplus_{k=1}^{n-1}
\bigoplus_{m=1}^{n-k}
\bigoplus_{\substack{i_{1}+\cdots+i_{k}=m\\i_{1},\ldots,i_{k}\geq0}}
\mathop\mathrm{tr}(\bm{A}^{i_{1}}\bm{B}\cdots\bm{A}^{i_{k}}\bm{B})
\oplus
\mathop\mathrm{Tr}(\bm{B}).
\label{E-TrAB}
\end{equation}

\section{Constrained Bi-Objective Optimization Problem}
\label{S-CBOOP}

In this section, we offer a complete Pareto-optimal solution to a constrained bi-objective optimization problem, which is formulated in terms of an arbitrary tropical semifield and solved under rather general conditions. Suppose that, given matrices $\bm{A},\bm{B},\bm{C}\in\mathbb{X}^{n\times n}$, the problem is to find regular vectors $\bm{x}\in\mathbb{X}^{n}$ that
\begin{equation}
\begin{aligned}
&
\text{minimize}
&&
(\bm{x}^{-}\bm{A}\bm{x},\ \bm{x}^{-}\bm{B}\bm{x});
\\
&
\text{subject to}
&&
\bm{C}\bm{x}
\leq
\bm{x}.
\label{P-minxAxxBx-Cxleqx}
\end{aligned}
\end{equation}

Before solving the problem, we note that, by Theorem~\ref{T-Axleqx}, the inequality constraint has nontrivial solutions only under the condition $\mathop\mathrm{Tr}(\bm{C})\leq\mathbb{1}$, which we have to take as a necessary assumption.

To describe the solution in a compact form, we use the following notation. For any matrices $\bm{A},\bm{B},\bm{C}\in\mathbb{X}^{n\times n}$, we denote the spectral radii of the matrices $\bm{A}$ and $\bm{B}$ as
\begin{equation}
\lambda
=
\bigoplus_{k=1}^{n}
\mathop\mathrm{tr}\nolimits^{1/k}(\bm{A}^{k}),
\qquad
\mu
=
\bigoplus_{k=1}^{n}
\mathop\mathrm{tr}\nolimits^{1/k}(\bm{B}^{k}),
\label{E-lambda-mu}
\end{equation}
and introduce the scalars
\begin{equation}
\begin{aligned}
\sigma
=
\bigoplus_{k=1}^{n-1}
\bigoplus_{m=1}^{n-k}
\bigoplus_{\substack{i_{1}+\cdots+i_{k}=m\\i_{1},\ldots,i_{k}\geq0}}
\mathop\mathrm{tr}\nolimits^{1/m}(\bm{A}^{i_{1}}\bm{C}\cdots\bm{A}^{i_{k}}\bm{C}),
\\
\theta
=
\bigoplus_{k=1}^{n-1}
\bigoplus_{m=1}^{n-k}
\bigoplus_{\substack{i_{1}+\cdots+i_{k}=m\\i_{1},\ldots,i_{k}\geq0}}
\mathop\mathrm{tr}\nolimits^{1/m}(\bm{B}^{i_{1}}\bm{C}\cdots\bm{B}^{i_{k}}\bm{C}).
\end{aligned}
\label{E-sigma-theta}
\end{equation}

Next, for any $k=1,\ldots,n-1$, $l=1,\ldots,k$ and $m=1,\ldots,n-k$, we denote
\begin{equation}
r_{k,l,m}
=
\bigoplus_{\substack{i_{1}+\cdots+i_{k}=m\\i_{1},\ldots,i_{k}\geq0}}
\bigoplus_{\substack{j_{1}+\cdots+j_{k}=l\\j_{1},\ldots,j_{k}\in\{0,1\}}}
\mathop\mathrm{tr}(\bm{A}^{i_{1}}\bm{B}^{j_{1}}\bm{C}^{1-j_{1}}\cdots\bm{A}^{i_{k}}\bm{B}^{j_{k}}\bm{C}^{1-j_{k}}).
\label{E-rklm}
\end{equation}

Finally, we define, for any $s,t>\mathbb{0}$, the functions
\begin{equation}
G(s)
=
\bigoplus_{k=1}^{n-1}
\bigoplus_{m=1}^{n-k}
\bigoplus_{l=1}^{k}
r_{k,l,m}^{1/l}
s^{-m/l},
\qquad
H(t)
=
\bigoplus_{k=1}^{n-1}
\bigoplus_{m=1}^{n-k}
\bigoplus_{l=1}^{k}
r_{k,l,m}^{1/m}
t^{-l/m}.
\label{E-Gs-Ht}
\end{equation}

We observe that both functions $G$ and $H$ monotonically decrease as their arguments increase. Moreover, the equality $G(s)=t$ is equivalent to $H(t)=s$, and thus these functions are inverse to each other. Suppose the equality $G(s)=t$ is valid. This equality is equivalent to the system
\begin{equation*}
r_{k,l,m}^{1/l}
s^{-m/l}
\leq
t,
\qquad
l=1,\ldots,k;
\quad
m=1,\ldots,n-k;
\quad
k=1,\ldots,n-1,
\end{equation*}
where at least one inequality holds as an equality. The solution of these inequalities for $s$~yields
\begin{equation*}
r_{k,l,m}^{1/m}
t^{-l/m}
\leq
s,
\qquad
l=1,\ldots,k;
\quad
m=1,\ldots,n-k;
\quad
k=1,\ldots,n-1.
\end{equation*}

By combining these inequalities, one of which is an equality, we obtain $H(t)=s$.

As a consequence, we also see that the inequalities $G(s)\leq t$ and $H(t)\leq s$ are dual in the sense that the solution of one of them is given by the other and vice versa.

We are now in a position to formulate and proof the main result of the paper.

\begin{theorem}
\label{T-minxAxxBx-Cxleqx}
Let $\bm{A}$ and $\bm{B}$ be matrices with respective spectral radii $\lambda>\mathbb{0}$ and $\mu>\mathbb{0}$, and $\bm{C}$ be a matrix with $\mathop\mathrm{Tr}(\bm{C})\leq\mathbb{1}$. Then, the following statements hold.
\begin{enumerate}
\item
If $H(\mu\oplus\theta)\leq\lambda\oplus\sigma$, then the Pareto frontier of problem \eqref{P-minxAxxBx-Cxleqx} reduces to the point $(\alpha,\beta)$ with
\begin{equation*}
\alpha
=
\lambda\oplus\sigma,
\qquad
\beta
=
\mu\oplus\theta.
\end{equation*}
\item
If $H(\mu\oplus\theta)>\lambda\oplus\sigma$, then the Pareto frontier is a segment defined by
\begin{equation*}
\lambda\oplus\sigma
\leq
\alpha
\leq
H(\mu\oplus\theta),
\qquad
\beta
=
G(\alpha).
\end{equation*}
\item
All Pareto-optimal solutions of problem \eqref{P-minxAxxBx-Cxleqx} are given in parametric form by
\begin{equation*}
\bm{x}
=
(\alpha^{-1}\bm{A}
\oplus
\beta^{-1}\bm{B}
\oplus
\bm{C})^{\ast}
\bm{u},
\qquad
\bm{u}
>
\bm{0}.
\end{equation*}
\end{enumerate} 
\end{theorem}

\begin{proof}
Fix an arbitrary solution vector $\bm{x}$ and denote the corresponding values of the objective functions $\bm{x}^{-}\bm{A}\bm{x}$ and $\bm{x}^{-}\bm{B}\bm{x}$ in the Pareto frontier of the problem by $\alpha$ and $\beta$ respectively. The solution $\bm{x}$ is then given by the parametrized system
\begin{equation*}
\bm{x}^{-}\bm{A}\bm{x}
=
\alpha,
\qquad
\bm{x}^{-}\bm{B}\bm{x}
=
\beta,
\qquad
\bm{C}\bm{x}
\leq
\bm{x}.
\end{equation*}

Since the parameters $\alpha$ and $\beta$ are assumed to take minimum values that cannot be improved, the set of solution vectors $\bm{x}$ does not change if the equalities are replaced by the inequalities $\bm{x}^{-}\bm{A}\bm{x}\leq\alpha$ and $\bm{x}^{-}\bm{B}\bm{x}\leq\beta$. Furthermore, we use Lemma~\ref{L-Axleqd} to solve the first inequality with respect to $\bm{A}\bm{x}$ and the second with respect to $\bm{B}\bm{x}$, and then rearrange the parameters. The system now becomes
\begin{equation*}
\alpha^{-1}\bm{A}\bm{x}
\leq
\bm{x},
\qquad
\beta^{-1}\bm{B}\bm{x}
\leq
\bm{x},
\qquad
\bm{C}\bm{x}
\leq
\bm{x}.
\end{equation*}

Finally, we combine the inequalities of the system into one parametrized inequality
\begin{equation*}
(\alpha^{-1}\bm{A}
\oplus
\beta^{-1}\bm{B}
\oplus
\bm{C})
\bm{x}
\leq
\bm{x}.
\end{equation*}

By Theorem~\ref{T-Axleqx}, this inequality has regular solutions if and only if the following condition holds:
\begin{equation}
\mathop\mathrm{Tr}(\alpha^{-1}\bm{A}\oplus\beta^{-1}\bm{B}\oplus\bm{C})
\leq
\mathbb{1}.
\label{I-Tralpha1Abeta1BCleq1}
\end{equation}

Under this condition, all regular solutions are given in the parametric form
\begin{equation}
\bm{x}
=
(\alpha^{-1}\bm{A}
\oplus
\beta^{-1}\bm{B}
\oplus
\bm{C})^{\ast}
\bm{u},
\qquad
\bm{u}
>
\bm{0}.
\label{E-xalpha1Abeta1BCu}
\end{equation}

Consider the existence condition at \eqref{I-Tralpha1Abeta1BCleq1} and apply identity \eqref{E-TrAB} to expand the left-hand side as 
\begin{multline*}
\mathop\mathrm{Tr}(\alpha^{-1}\bm{A}\oplus\beta^{-1}\bm{B}\oplus\bm{C})
=
\bigoplus_{k=1}^{n}
\alpha^{-k}\mathop\mathrm{tr}\bm{A}^{k}
\oplus
\mathop\mathrm{Tr}(\beta^{-1}\bm{B}\oplus\bm{C})
\\\oplus
\bigoplus_{k=1}^{n-1}
\bigoplus_{m=1}^{n-k}
\bigoplus_{\substack{i_{1}+\cdots+i_{k}=m\\i_{1},\ldots,i_{k}\geq0}}
\alpha^{-m}
\mathop\mathrm{tr}(\bm{A}^{i_{1}}(\beta^{-1}\bm{B}\oplus\bm{C})\cdots\bm{A}^{i_{k}}(\beta^{-1}\bm{B}\oplus\bm{C})).
\end{multline*} 

Then, the existence condition is equivalent to the system of inequalities
\begin{equation}
\begin{aligned}
\bigoplus_{k=1}^{n}
\alpha^{-k}\mathop\mathrm{tr}\bm{A}^{k}
\leq
\mathbb{1},
\qquad
\mathop\mathrm{Tr}(\beta^{-1}\bm{B}\oplus\bm{C})
\leq
\mathbb{1},
\\
\bigoplus_{k=1}^{n-1}
\bigoplus_{m=1}^{n-k}
\bigoplus_{\substack{i_{1}+\cdots+i_{k}=m\\i_{1},\ldots,i_{k}\geq0}}
\alpha^{-m}
\mathop\mathrm{tr}(\bm{A}^{i_{1}}(\beta^{-1}\bm{B}\oplus\bm{C})\cdots\bm{A}^{i_{k}}(\beta^{-1}\bm{B}\oplus\bm{C}))
\leq
\mathbb{1}.
\end{aligned}
\label{I-alpha-beta}
\end{equation}

We now solve these inequalities with respect to the parameters $\alpha$ and $\beta$. The first inequality can be replaced by the inequalities
\begin{equation*}
\alpha^{-k}\mathop\mathrm{tr}\bm{A}^{k}
\leq
\mathbb{1},
\qquad
k=1,\ldots,n.
\end{equation*}

After solving these inequalities for $\alpha$ and combining the results, we have
\begin{equation*}
\alpha
\geq
\bigoplus_{k=1}^{n}
\mathop\mathrm{tr}\nolimits^{1/k}(\bm{A}^{k})
=
\lambda.
\end{equation*}

Next, we examine the second inequality at \eqref{I-alpha-beta}. An application of \eqref{E-TrAB} to the left-hand side yields 
\begin{multline*}
\mathop\mathrm{Tr}(\beta^{-1}\bm{B}\oplus\bm{C})
=
\bigoplus_{k=1}^{n}
\beta^{-k}\mathop\mathrm{tr}\bm{B}^{k}
\\\oplus
\bigoplus_{k=1}^{n-1}
\bigoplus_{m=1}^{n-k}
\bigoplus_{\substack{i_{1}+\cdots+i_{k}=m\\i_{1},\ldots,i_{k}\geq0}}
\beta^{-m}
\mathop\mathrm{tr}(\bm{B}^{i_{1}}\bm{C}\cdots\bm{B}^{i_{k}}\bm{C})
\oplus
\mathop\mathrm{Tr}(\bm{C}).
\end{multline*}

The inequality under examination reduces to the system
\begin{equation*}
\begin{aligned}
\bigoplus_{k=1}^{n}
\beta^{-k}\mathop\mathrm{tr}\bm{B}^{k}
\leq
\mathbb{1},
\qquad
\mathop\mathrm{Tr}(\bm{C})
\leq
\mathbb{1},
\\
\bigoplus_{k=1}^{n-1}
\bigoplus_{m=1}^{n-k}
\bigoplus_{\substack{i_{1}+\cdots+i_{k}=m\\i_{1},\ldots,i_{k}\geq0}}
\beta^{-m}
\mathop\mathrm{tr}(\bm{B}^{i_{1}}\bm{C}\cdots\bm{B}^{i_{k}}\bm{C})
\leq
\mathbb{1},
\end{aligned}
\end{equation*}
where the inequality $\mathop\mathrm{Tr}(\bm{C})\leq\mathbb{1}$ holds by the assumption of the theorem.

Solving the first and third inequality in the same way as above, we obtain
\begin{align*}
\beta
&\geq
\bigoplus_{k=1}^{n}
\mathop\mathrm{tr}\nolimits^{1/k}(\bm{B}^{k})
=
\mu,
\\
\beta
&\geq
\bigoplus_{k=1}^{n-1}
\bigoplus_{m=1}^{n-k}
\bigoplus_{\substack{i_{1}+\cdots+i_{k}=m\\i_{1},\ldots,i_{k}\geq0}}
\mathop\mathrm{tr}\nolimits^{1/m}(\bm{B}^{i_{1}}\bm{C}\cdots\bm{B}^{i_{k}}\bm{C})
=
\theta.
\end{align*}

Finally, we consider the matrix product $\bm{A}^{i_{1}}(\beta^{-1}\bm{B}\oplus\bm{C})\cdots\bm{A}^{i_{k}}(\beta^{-1}\bm{B}\oplus\bm{C})$ under the trace operator on the left-hand side of the third inequality at \eqref{I-alpha-beta}. We expand the product by taking the first and then the second summand in each term $\beta^{-1}\bm{B}\oplus\bm{C}$ to obtain the sum of products $\beta^{-(j_{1}+\cdots+j_{k})}(\bm{A}^{i_{1}}\bm{B}^{j_{1}}\bm{C}^{1-j_{1}}\cdots\bm{A}^{i_{k}}\bm{B}^{j_{k}}\bm{C}^{1-j_{k}})$ for all $j_{1},\ldots,j_{k}\in\{0,1\}$ such that $0\leq j_{1}+\cdots+j_{k}\leq k$.

The trace under summation on the left-hand side of the inequality takes the form
\begin{multline*}
\mathop\mathrm{tr}(\bm{A}^{i_{1}}(\beta^{-1}\bm{B}\oplus\bm{C})\cdots\bm{A}^{i_{k}}(\beta^{-1}\bm{B}\oplus\bm{C}))
=
\mathop\mathrm{tr}(\bm{A}^{i_{1}}\bm{C}\cdots\bm{A}^{i_{k}}\bm{C})
\\\oplus
\bigoplus_{l=1}^{k}
\bigoplus_{\substack{j_{1}+\cdots+j_{k}=l\\j_{1},\ldots,j_{k}\in\{0,1\}}}
\beta^{-l}
\mathop\mathrm{tr}(\bm{A}^{i_{1}}\bm{B}^{j_{1}}\bm{C}^{1-j_{1}}\cdots\bm{A}^{i_{k}}\bm{B}^{j_{k}}\bm{C}^{1-j_{k}}).
\end{multline*}

After substitution of this expression, we use the symbol $r_{k,l,m}$ to rewrite the inequality as the system of two inequalities
\begin{align*}
\bigoplus_{k=1}^{n-1}
\bigoplus_{m=1}^{n-k}
\bigoplus_{\substack{i_{1}+\cdots+i_{k}=m\\i_{1},\ldots,i_{k}\geq0}}
\alpha^{-m}
\mathop\mathrm{tr}(\bm{A}^{i_{1}}\bm{C}\cdots\bm{A}^{i_{k}}\bm{C})
&\leq
\mathbb{1},
\\
\bigoplus_{k=1}^{n-1}
\bigoplus_{m=1}^{n-k}
\bigoplus_{l=1}^{k}
r_{k,l,m}
\alpha^{-m}
\beta^{-l}
&\leq
\mathbb{1}.
\end{align*}

The solution of the first inequality with respect to $\alpha$ and the second to $\beta$ yields
\begin{align*}
\alpha
&\geq
\bigoplus_{k=1}^{n-1}
\bigoplus_{m=1}^{n-k}
\bigoplus_{\substack{i_{1}+\cdots+i_{k}=m\\i_{1},\ldots,i_{k}\geq0}}
\mathop\mathrm{tr}\nolimits^{1/m}(\bm{A}^{i_{1}}\bm{C}\cdots\bm{A}^{i_{k}}\bm{C})
=
\sigma,
\\
\beta
&\geq
\bigoplus_{k=1}^{n-1}
\bigoplus_{m=1}^{n-k}
\bigoplus_{l=1}^{k}
r_{k,l,m}^{1/l}
\alpha^{-m/l}
=
G(\alpha).
\end{align*}

Finally, we combine all inequalities obtained for $\alpha$ and $\beta$ into the following system:
\begin{equation}
\alpha
\geq
\lambda
\oplus
\sigma,
\qquad
\beta
\geq
\mu
\oplus
\theta
\oplus
G(\alpha).
\label{I-alphageqlambdasigma-betageqmutheta}
\end{equation}

Taking into account that $G(\alpha)$ decreases as $\alpha$ increases, we see that this system defines an area on the $\alpha\beta$-plane that is bounded from the left and from below by the lines $\alpha=
\lambda\oplus\sigma$ and $\beta=\mu\oplus\theta$, and lies above the graph of the function $\beta=G(\alpha)$.

To describe the Pareto frontier for the problem, we note that any interior point of the area can be improved and hence cannot belong to the frontier. For the same reason, the left and bottom boundary half-lines cannot be parts of the frontier, except for their bottommost and leftmost points. Both points may coincide in the single point $(\lambda\oplus\sigma,\mu\oplus\theta)$ if the graph of $G(\alpha)$ lies below this point, or be the ends of a segment that is cut out from the graph by the lines $\alpha=\lambda\oplus\sigma$ and $\beta=\mu\oplus\theta$ otherwise.

An illustration is given in Figure~\ref{F-EPF}, where the frontier is indicated by a thick dot (left) or depicted by a thick segment between two thick dots of a curve (right).
\begin{figure}[ht]
\begin{tikzpicture}

\begin{axis}[
axis lines=none,
]

\addplot[
black,
line width=1.0pt,
domain=1:50,
y domain=1:10,
]
{max(1/x^(3/2),1/x,1/x^(1/2)/50};
\end{axis}

\draw (0.5,0.4) -- (6.5,0.4) -- (6.5,6.4) -- (0.5,6.4) -- (0.5,0.4);
\node at (0.0,3.5) {$\beta$};
\node at (3.5,0.0) {$\alpha$};

\draw [thick] (3.0,0.4) -- (3.0,6.4);
\draw [thick] (0.5,1.6) -- (6.5,1.6);


\fill (3.0,1.6) circle (2pt);

\node at (4.0,6.0) {$\alpha=\lambda\oplus\sigma$};

\node at (5.5,1.9) {$\beta=\mu\oplus\theta$};

\node at (1.8,1.9) {$H(\mu\oplus\theta)$};

\node at (4.5,0.9) {$\beta=G(\alpha)$};

\end{tikzpicture}
\hspace{1mm}
\begin{tikzpicture}

\begin{axis}[
axis lines=none,
]

\addplot[
black,
line width=1.0pt,
domain=1:4,
y domain=1:10,
]
{max(1/x^(4),1/x,1/x^(1/3)/8};

\addplot[
black,
line width=1.75pt,
domain=1.5:2.05,
y domain=1:10,
]
{max(1/x^(4),1/x,1/x^(1/3)/8};

\end{axis}

\draw (0.5,0.4) -- (6.5,0.4) -- (6.5,6.4) -- (0.5,6.4) -- (0.5,0.4);
\node at (0.0,3.5) {$\beta$};
\node at (3.5,0.0) {$\alpha$};

\draw [thick] (1.5,0.4) -- (1.5,6.4);
\draw [thick] (0.5,2.0) -- (6.5,2.0);


\fill (1.5,3.15) circle (2pt);
\fill (2.55,2.0) circle (2pt);



\node at (2.5,6.0) {$\alpha=\lambda\oplus\sigma$};

\node at (5.5,2.3) {$\beta=\mu\oplus\theta$};

\node at (3.45,2.3) {$H(\mu\oplus\theta)$};

\node at (2.9,0.9) {$\beta=G(\alpha)$};
\end{tikzpicture}
\caption{Examples of Pareto frontier in the form of a single point (\textbf{left}), and of a segment (\textbf{right}).}
\label{F-EPF}
\end{figure}
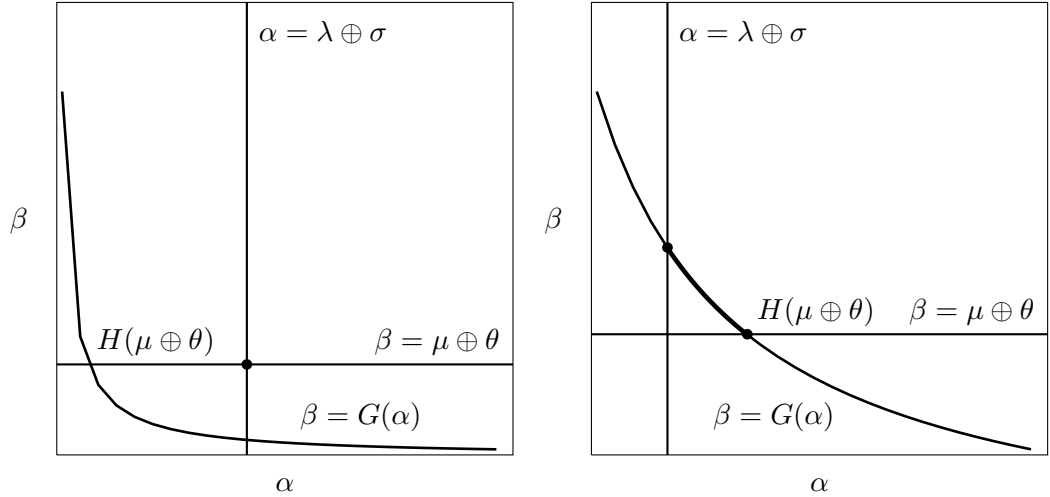

We consider two cases and initially suppose the following condition holds:
\begin{equation*}
H(\mu\oplus\theta)
\leq
\lambda\oplus\sigma.
\end{equation*}

This condition and the first inequality at \eqref{I-alphageqlambdasigma-betageqmutheta} lead to the inequality $H(\mu\oplus\theta)\leq\alpha$, which is equivalent to the dual inequality $\mu\oplus\theta\geq G(\alpha)$. The second inequality at \eqref{I-alphageqlambdasigma-betageqmutheta} reduces to $\beta\geq\mu\oplus\theta$. As a result, the Pareto frontier shrinks to a single point where the parameters are fixed at
\begin{equation*}
\alpha
=
\lambda\oplus\sigma,
\qquad
\beta
=
\mu\oplus\theta.
\end{equation*}

Substitution of these values for the parameters into \eqref{E-xalpha1Abeta1BCu} yields all corresponding Pareto-optimal solutions of the problem in this case.

We now examine the case that
\begin{equation*}
H(\mu\oplus\theta)
>
\lambda\oplus\sigma.
\end{equation*}

Assume the parameter $\alpha$ to satisfy the inequality $\lambda\oplus\sigma\leq\alpha<H(\mu\oplus\theta)$, and note that the condition $\alpha<H(\mu\oplus\theta)$ results in the inequality $\mu\oplus\theta<G(\alpha)$. Indeed, if the opposite inequality $\mu\oplus\theta\geq G(\alpha)$ holds, then $\alpha\geq H(\mu\oplus\theta)$, which is a contradiction.

Since the second inequality at \eqref{I-alphageqlambdasigma-betageqmutheta} then becomes $\beta\geq G(\alpha)$, the Pareto frontier forms a non-linear segment given by $\lambda\oplus\sigma\leq\alpha<H(\mu\oplus\theta)$ and $\beta=G(\alpha)$.

For all $\alpha\geq H(\mu\oplus\theta)$, the equality $\mu\oplus\theta\geq G(\alpha)$ holds, and thus $\beta\geq\mu\oplus\theta$. The Pareto frontier degenerates to the point with $\alpha=H(\mu\oplus\theta)$ and $\beta=\mu\oplus\theta$.

Observing that $G(\alpha)=G(H(\mu\oplus\theta))=\mu\oplus\theta$, we combine the results to represent the frontier as 
\begin{equation*}
\lambda\oplus\sigma
\leq
\alpha
\leq
H(\mu\oplus\theta),
\qquad
\beta
=
G(\alpha).
\end{equation*}

A complete Pareto-optimal solution takes the general form defined by \eqref{E-xalpha1Abeta1BCu}. 
\end{proof}
 
Let us set $\bm{C}=\bm{0}$ in problem \eqref{P-minxAxxBx-Cxleqx}, which makes the constraint $\bm{C}\bm{x}\leq\bm{x}$ trivially hold for any $\bm{x}$. As a result, this problem becomes the unconstrained problem
\begin{equation}
\begin{aligned}
&
\text{minimize}
&&
(\bm{x}^{-}\bm{A}\bm{x},\ \bm{x}^{-}\bm{B}\bm{x}).
\label{P-minxAxxBx}
\end{aligned}
\end{equation}

A complete solution of problem \eqref{P-minxAxxBx} can be obtained as a formal consequence of the solution of \eqref{P-minxAxxBx-Cxleqx} as follows. First note that, under the condition $\bm{C}=\bm{0}$, we have $\sigma=\mathbb{0}$ and $\theta=\mathbb{0}$.

Furthermore, we see that $r_{k,l,m}\ne\mathbb{0}$ only when $l=k$. Then, for any $m=1,\ldots,n-k$ and $k=1,\ldots,n$, we introduce $r_{k,m}$ equal to $r_{k,k,m}$, and write 
\begin{equation*}
r_{k,m}
=
\bigoplus_{\substack{i_{1}+\cdots+i_{k}=m\\i_{1},\ldots,i_{k}\geq0}}
\mathop\mathrm{tr}(\bm{A}^{i_{1}}\bm{B}\cdots\bm{A}^{i_{k}}\bm{B}).
\end{equation*}

Finally, we redefine, for any $s,t>\mathbb{0}$, the functions
\begin{equation*}
G(s)
=
\bigoplus_{k=1}^{n-1}
\bigoplus_{m=1}^{n-k}
r_{k,m}^{1/k}
s^{-m/k},
\qquad
H(t)
=
\bigoplus_{k=1}^{n-1}
\bigoplus_{m=1}^{n-k}
r_{k,m}^{1/m}
t^{-k/m}.
\end{equation*}

With the condition $\bm{C}=\bm{0}$ and new notation, the solution of problem \eqref{P-minxAxxBx-Cxleqx} given by Theorem~\ref{T-minxAxxBx-Cxleqx} reduces to a solution of \eqref{P-minxAxxBx} in the form of the next result.
\begin{corollary}
\label{C-minxAxxBx}
Let $\bm{A}$ and $\bm{B}$ be matrices with respective spectral radii $\lambda>\mathbb{0}$ and $\mu>\mathbb{0}$. Then, the following statements hold.
\begin{enumerate}
\item
If $H(\mu)\leq\lambda$, then the Pareto frontier of problem \eqref{P-minxAxxBx} is a point $(\alpha,\beta)$ with
\begin{equation*}
\alpha
=
\lambda,
\qquad
\beta
=
\mu.
\end{equation*}
\item
If $H(\mu)>\lambda$, then the Pareto frontier is a segment defined by
\begin{equation*}
\lambda
\leq
\alpha
\leq
H(\mu),
\qquad
\beta
=
G(\alpha).
\end{equation*}
\item
All Pareto-optimal solutions of problem \eqref{P-minxAxxBx} are given in parametric form by
\begin{equation*}
\bm{x}
=
(\alpha^{-1}\bm{A}
\oplus
\beta^{-1}\bm{B})^{\ast}
\bm{u},
\qquad
\bm{u}
>
\bm{0}.
\end{equation*}
\end{enumerate} 
\end{corollary}

This solution corresponds to that given in \cite{Krivulin2020Using} for the unconstrained problem. 

We conclude this section with an example of solution of a general two-dimensional problem, which illustrates in more detail the results obtained.
\begin{example}
\label{X-Aeqa11a12a21a22}
Suppose that one needs to find a vector $\bm{x}=(x_{1},x_{2})^{T}$ that solves problem \eqref{P-minxAxxBx-Cxleqx} with $n=2$ and the matrices given by
\begin{equation*}
\bm{A}
=
\begin{pmatrix}
a_{11} & a_{12}
\\
a_{21} & a_{22}
\end{pmatrix},
\qquad
\bm{B}
=
\begin{pmatrix}
b_{11} & b_{12}
\\
b_{21} & b_{22}
\end{pmatrix},
\qquad
\bm{C}
=
\begin{pmatrix}
c_{11} & c_{12}
\\
c_{21} & c_{22}
\end{pmatrix}.
\end{equation*}

To apply Theorem~\ref{T-minxAxxBx-Cxleqx}, we need to reformulate the assumptions and statements of the theorem for the two-dimensional case. First, we calculate the matrix 
\begin{equation*}
\bm{A}^{2}
=
\begin{pmatrix}
a_{11}^{2}\oplus a_{12}a_{21} & a_{12}(a_{11}\oplus a_{22})
\\
a_{21}(a_{11}\oplus a_{22}) & a_{12}a_{21}\oplus a_{22}^{2}
\end{pmatrix},
\end{equation*}
and then use \eqref{E-lambda-mu} to obtain the spectral radius of the matrix $\bm{A}$ as
\begin{equation*}
\lambda
=
\mathop\mathrm{tr}\bm{A}
\oplus
\mathop\mathrm{tr}\nolimits^{1/2}(\bm{A}^{2})
=
a_{11}\oplus a_{22}\oplus a_{12}^{1/2}a_{21}^{1/2}.
\end{equation*}

The spectral radius of $\bm{B}$ is similarly given by
\begin{equation*}
\mu
=
\mathop\mathrm{tr}\bm{B}
\oplus
\mathop\mathrm{tr}\nolimits^{1/2}(\bm{B}^{2})
=
b_{11}\oplus b_{22}\oplus b_{12}^{1/2}b_{21}^{1/2}.
\end{equation*}

Furthermore, with the matrix
\begin{equation*}
\bm{C}^{2}
=
\begin{pmatrix}
c_{11}^{2}\oplus c_{12}c_{21} & c_{12}(c_{11}\oplus c_{22})
\\
c_{21}(c_{11}\oplus c_{22}) & c_{12}c_{21}\oplus c_{22}^{2}
\end{pmatrix},
\end{equation*}
we derive the tropical determinant of the matrix $\bm{C}$ in the form
\begin{equation*}
\mathop\mathrm{Tr}(\bm{C})
=
\mathop\mathrm{tr}\bm{C}
\oplus
\mathop\mathrm{tr}\bm{C}^{2}
=
c_{11}\oplus c_{11}^{2}\oplus c_{22}\oplus c_{22}^{2}\oplus c_{12}c_{21}.
\end{equation*}

The assumptions of the theorem to be made about the matrices now take the form
\begin{equation*}
a_{11}\oplus a_{22}\oplus a_{12}^{1/2}a_{21}^{1/2}
>\mathbb{0},
\qquad
b_{11}\oplus b_{22}\oplus b_{12}^{1/2}b_{21}^{1/2}
>\mathbb{0},
\qquad
c_{11}\oplus c_{22}\oplus c_{12}c_{21}
\leq
\mathbb{1}.
\end{equation*}

Next, we adjust the statements of Theorem~\ref{T-minxAxxBx-Cxleqx}. To describe the conditions, we evaluate the~matrix
\begin{equation*}
\bm{A}\bm{C}
=
\begin{pmatrix}
a_{11}c_{11}\oplus a_{12}c_{21} & a_{11}c_{12}\oplus a_{12}c_{22}
\\
a_{21}c_{11}\oplus a_{22}c_{21} & a_{21}c_{12}\oplus a_{22}c_{22}
\end{pmatrix},
\end{equation*}
and then apply \eqref{E-sigma-theta} to write
\begin{equation*}
\sigma
=
\mathop\mathrm{tr}(\bm{A}\bm{C})
=
a_{11}c_{11}\oplus a_{12}c_{21}
\oplus
a_{21}c_{12}\oplus a_{22}c_{22}.
\end{equation*}

In the same way, the evaluation of the matrix $\bm{B}\bm{C}$ yields
\begin{equation*}
\theta
=
\mathop\mathrm{tr}(\bm{B}\bm{C})
=
b_{11}c_{11}\oplus b_{12}c_{21}
\oplus
b_{21}c_{12}\oplus b_{22}c_{22}.
\end{equation*}

The set of scalars defined by \eqref{E-rklm} reduces to one scalar
\begin{equation*}
r_{1,1,1}
=
\mathop\mathrm{tr}(\bm{A}\bm{B})
=
a_{11}b_{11}\oplus a_{12}b_{21}
\oplus
a_{21}b_{12}\oplus a_{22}b_{22},
\end{equation*}
which converts the functions at \eqref{E-Gs-Ht} into
\begin{equation*}
G(s)
=
s^{-1}\mathop\mathrm{tr}(\bm{A}\bm{B}),
\qquad
H(t)
=
t^{-1}\mathop\mathrm{tr}(\bm{A}\bm{B}).
\end{equation*}

It remains to refine the Kleene star matrix, which generates the solutions of the problem. Observing the assumption of the theorem, we have $c_{11},c_{22}\leq\mathop\mathrm{Tr}(\bm{C})\leq\mathbb{1}$. Moreover, it follows from the statement of the theorem that $\alpha\geq\lambda\oplus\sigma\geq a_{ii}$ and $\beta\geq\mu\oplus\theta\geq b_{ii}$ for $i=1,2$. Then, we can represent the matrix as follows:
\begin{multline*}
(\alpha^{-1}\bm{A}\oplus\beta^{-1}\bm{B}\oplus\bm{C})^{\ast}
=
\bm{I}
\oplus
\alpha^{-1}\bm{A}\oplus\beta^{-1}\bm{B}\oplus\bm{C}
\\=
\begin{pmatrix}
\mathbb{1} & \alpha^{-1}a_{12}\oplus\beta^{-1}b_{12}\oplus c_{12}
\\
\alpha^{-1}a_{21}\oplus\beta^{-1}b_{21}\oplus c_{21} & \mathbb{1}
\end{pmatrix}.
\end{multline*}

To rewrite the results of Theorem~\ref{T-minxAxxBx-Cxleqx} in terms of the current problem, we consider two cases. If the condition $\mathop\mathrm{tr}(\bm{A}\bm{B})\leq(\lambda\oplus\mathop\mathrm{tr}(\bm{A}\bm{C}))(\mu\oplus\mathop\mathrm{tr}(\bm{B}\bm{C}))$ holds, then the Pareto frontier of the problem is the point $(\alpha,\beta)$ with
\begin{equation*}
\alpha
=
\lambda\oplus\mathop\mathrm{tr}(\bm{A}\bm{C}),
\qquad
\beta
=
\mu\oplus\mathop\mathrm{tr}(\bm{B}\bm{C}).
\end{equation*}

Otherwise, provided that $\mathop\mathrm{tr}(\bm{A}\bm{B})>(\lambda\oplus\mathop\mathrm{tr}(\bm{A}\bm{C}))(\mu\oplus\mathop\mathrm{tr}(\bm{B}\bm{C}))$, the Pareto frontier forms a segment that is defined as
\begin{equation*}
\lambda\oplus\mathop\mathrm{tr}(\bm{A}\bm{C})
\leq
\alpha
\leq
\mathop\mathrm{tr}(\bm{A}\bm{B})(\mu\oplus\mathop\mathrm{tr}(\bm{B}\bm{C}))^{-1},
\qquad
\beta
=
\alpha^{-1}\mathop\mathrm{tr}(\bm{A}\bm{B}).
\end{equation*}

The Pareto-optimal solution is given, using a parameter vector $\bm{u}=(u_{1},u_{2})^{T}$, by
\begin{equation*}
\bm{x}
=
\begin{pmatrix}
\mathbb{1} & \alpha^{-1}a_{12}\oplus\beta^{-1}b_{12}\oplus c_{12}
\\
\alpha^{-1}a_{21}\oplus\beta^{-1}b_{21}\oplus c_{21} & \mathbb{1}
\end{pmatrix}\bm{u},
\qquad
\bm{u}>\bm{0}.
\end{equation*}

Finally, assume that $\bm{A}$ and $\bm{B}$ are symmetrically reciprocal matrices, and $\bm{C}$ is an upper triangular matrix, given in the framework of the max-algebra $\mathbb{R}_{\max}$ by
\begin{equation*}
\bm{A}
=
\begin{pmatrix}
1 & 2
\\
1/2 & 1
\end{pmatrix},
\qquad
\bm{B}
=
\begin{pmatrix}
1 & 1/3
\\
3 & 1
\end{pmatrix},
\qquad
\bm{C}
=
\begin{pmatrix}
0 & 1
\\
0 & 0
\end{pmatrix}.
\end{equation*}

We see that the assumptions of Theorem~\ref{T-minxAxxBx-Cxleqx} are fulfilled because
\begin{equation*}
\bm{A}^{2}
=
\bm{A},
\qquad
\lambda
=
1,
\qquad
\bm{B}^{2}
=
\bm{B},
\qquad
\mu
=
1,
\qquad
\bm{C}^{2}
=
\bm{0},
\qquad
\mathop\mathrm{Tr}(\bm{C})
=
0.
\end{equation*}

Furthermore, we calculate the matrices
\begin{equation*}
\bm{A}\bm{B}
=
\begin{pmatrix}
6 & 2
\\
3 & 1
\end{pmatrix},
\qquad
\bm{A}\bm{C}
=
\begin{pmatrix}
0 & 1
\\
0 & 1/2
\end{pmatrix},
\qquad
\bm{B}\bm{C}
=
\begin{pmatrix}
0 & 1
\\
0 & 3
\end{pmatrix},
\end{equation*}
and then find their traces
\begin{equation*}
\mathop\mathrm{tr}(\bm{A}\bm{B})
=
6,
\qquad
\mathop\mathrm{tr}(\bm{A}\bm{C})
=
1/2,
\qquad
\mathop\mathrm{tr}(\bm{B}\bm{C})
=
3.
\end{equation*}

Since the condition $\mathop\mathrm{tr}(\bm{A}\bm{B})=6>(\lambda\oplus\mathop\mathrm{tr}(\bm{A}\bm{C}))(\mu\oplus\mathop\mathrm{tr}(\bm{B}\bm{C}))=3$ holds, the Pareto frontier of the problem is the segment
\begin{equation*} 
1
\leq
\alpha
\leq
2,
\qquad
\beta
=
6\alpha^{-1}.
\end{equation*} 

All Pareto-optimal solutions are given by
\begin{equation*}
\bm{x}
=
\begin{pmatrix}
1 & 2\alpha^{-1}\oplus3^{-1}\beta^{-1}\oplus1
\\
2^{-1}\alpha^{-1}\oplus3\beta^{-1}\oplus0 & 1
\end{pmatrix}\bm{u},
\qquad
\bm{u}>\bm{0}.
\end{equation*}

After substitution of $\beta=6\alpha^{-1}$ and some algebra, the generating matrix becomes
\begin{equation*}
\begin{pmatrix}
1 & 2\alpha^{-1}
\\
2^{-1}\alpha & 1
\end{pmatrix}.
\end{equation*}

Observing that both columns in this matrix are collinear, we use the first one to represent all solutions as 
\begin{equation*}
\bm{x}
=
\begin{pmatrix}
1
\\
2^{-1}\alpha
\end{pmatrix}u,
\qquad
u>0,
\qquad
1
\leq
\alpha
\leq
2.
\end{equation*}

Figure~\ref{F-PFPOS} offers a graphical illustration of the Pareto frontier shown by a thick segment (left), and the Pareto-optimal solution given by the cone formed by the vectors $(1,1/2)^{T}$ and $(1,1)^{T}$, which correspond to the end points of the frontier (right). 
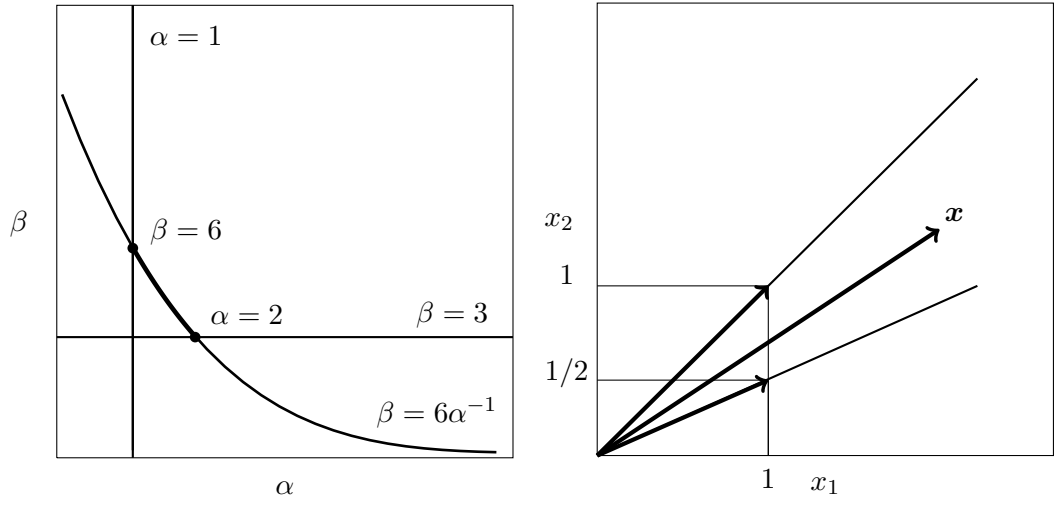
\begin{figure}[ht]
\begin{tikzpicture}

\begin{axis}[
axis lines=none,
]

\addplot[
black,
line width=1.0pt,
domain=0.5:0.9,
y domain=0:10,
]
{(6*x-6)^4};

\addplot[
black,
line width=1.75pt,
domain=0.565:0.625,
y domain=0:10,
]
{(6*x-6)^4};

\end{axis}

\draw (0.5,0.4) -- (6.5,0.4) -- (6.5,6.4) -- (0.5,6.4) -- (0.5,0.4);
\node at (0.0,3.5) {$\beta$};
\node at (3.5,0.0) {$\alpha$};

\draw [thick] (1.5,0.5) -- (1.5,6.4);
\draw [thick] (1.5,0.4) -- (1.5,6.4);
\draw [thick] (0.5,2.0) -- (6.5,2.0);


\fill (1.5,3.18) circle (2pt);
\fill (2.32,2.0) circle (2pt);



\node at (2.2,6.0) {$\alpha=1$};

\node at (2.2,3.4) {$\beta=6$};

\node at (5.7,2.3) {$\beta=3$};

\node at (3.0,2.3) {$\alpha=2$};

\node at (5.5,1.0) {$\beta=6\alpha^{-1}$};
\end{tikzpicture}
\hspace{1mm}
\begin{tikzpicture}

\draw (0.5,0.4) -- (6.5,0.4) -- (6.5,6.4) -- (0.5,6.4) -- (0.5,0.4);
\node at (0.0,3.5) {$x_{2}$};
\node at (3.5,0.0) {$x_{1}$};

\draw [thick] (0.5,0.4) -- (5.5,5.4);
\draw [ultra thick,->] (0.5,0.4) -- (2.75,2.65);

\draw [thick] (0.5,0.4) -- (5.5,2.65);
\draw [ultra thick,->] (0.5,0.4) -- (2.75,1.4);

\draw (0.5,2.65) -- (2.75,2.65);
\draw (2.75,0.4) -- (2.75,2.65);

\draw (2.75,0.4) -- (2.75,1.4);
\draw (0.5,1.4) -- (2.75,1.4);

\node at (0.1,2.8) {$1$};
\node at (0.1,1.5) {$1/2$};

\node at (2.75,0.1) {$1$};

\draw [ultra thick,->] (0.5,0.4) -- (5.0,3.4);

\node at (5.2,3.6) {$\bm{x}$};









\end{tikzpicture}
\caption{Pareto frontier (\textbf{left}) and Pareto-optimal solution (\textbf{right}).}
\label{F-PFPOS}
\end{figure}
\end{example}

\section{Application to Constrained Bi-Criteria Decision Problem}
\label{S-ACBCDP}

Consider the bi-criteria decision problem at \eqref{P-min_maxijaijxixj-maxijbijxixj-maxijxjleqxi}, and represent it in terms of the max-algebra $\mathbb{R}_{\max}$ as follows:
\begin{equation*}
\begin{aligned}
&
\text{minimize}
&&
\left(
\bigoplus_{1\leq i,j\leq n}x_{i}^{-1}a_{ij}x_{j},
\bigoplus_{1\leq i,j\leq n}x_{i}^{-1}b_{ij}x_{j}
\right);
\\
&
\text{subject to}
&&
\bigoplus_{1\leq j\leq n}
c_{ij}x_{j}
\leq
x_{i},
\quad
i=1,\ldots,n.
\end{aligned}
\end{equation*}

With the notation $\bm{A}=(a_{ij})$, $\bm{B}=(b_{ij})$, $\bm{C}=(c_{ij})$ and $\bm{x}=(x_{j})$, the vector objective function to minimize becomes
\begin{equation*}
(\bm{x}^{-}\bm{A}\bm{x},\ \bm{x}^{-}\bm{B}\bm{x}),
\end{equation*}
whereas the inequality constraints can be written as
\begin{equation*}
\bm{C}\bm{x}
\leq
\bm{x}.
\end{equation*}

After combining the objective and constraint, the problem takes the form of the optimization problem at \eqref{P-minxAxxBx-Cxleqx}, and thus has the solution given by Theorem~\ref{T-minxAxxBx-Cxleqx}. 

Note that the conditions in the theorem on spectral radii $\lambda>0$ and $\mu>0$ are trivially fulfilled for pairwise comparison matrices. The condition $\mathop\mathrm{Tr}(\bm{C})\leq1$ implies that the system of inequality constraints on ratings has nontrivial (positive) solutions.

We now present an example of application of Theorem~\ref{T-minxAxxBx-Cxleqx} to a constrained bi-criteria decision problem with four alternatives. As a basis of the example, we take an unconstrained problem in \cite{Krivulin2020Using} to use intermediate results of this problem to save writing. Then, we add constraints on ratings of alternatives, and construct a Pareto-optimal solution of the constrained problem obtained.

\begin{example}
Consider a problem to evaluate ratings of $n=4$ alternatives from pairwise comparisons according to two criteria subject to constraints on the ratings. The matrices of comparisons and constraints are given by
\begin{gather*}
\bm{A}
=
\begin{pmatrix}
1 & 3 & 4 & 2
\\
1/3 & 1 & 1/2 & 1/3
\\
1/4 & 2 & 1 & 4
\\
1/2 & 3 & 1/4 & 1
\end{pmatrix},
\qquad
\bm{B}
=
\begin{pmatrix}
1 & 2 & 4 & 2
\\
1/2 & 1 & 1/3 & 1/2
\\
1/4 & 3 & 1 & 4
\\
1/2 & 2 & 1/4 & 1
\end{pmatrix},
\\
\bm{C}
=
\begin{pmatrix}
0 & 0 & 0 & 0
\\
0 & 0 & 0 & 1
\\
0 & 0 & 0 & 0
\\
0 & 0 & 0 & 0
\end{pmatrix}.
\end{gather*}

The unconstrained version of the problem (with $\bm{C}=\bm{0}$) has a complete solution obtained in~\cite{Krivulin2020Using}, where the Pareto frontier is derived as the segment
\begin{equation*}
2
\leq
\alpha
\leq
3,
\qquad
\beta
=
24\alpha^{-3}\oplus24^{1/3}\alpha^{-1/3}.
\end{equation*} 

The corresponding Pareto-optimal solutions are given in the parametric form 
\begin{equation*}
\bm{x}
=
(\alpha^{-1}\bm{A}\oplus\beta^{-1}\bm{B})^{\ast}
\bm{u},
\qquad
\bm{u}
>
\bm{0},
\end{equation*}
which reduces, on the ends of the frontier with $\alpha$ set to $\alpha_{1}=2$ and $\alpha_{2}=3$, to the vectors
\begin{equation*}
\bm{x}_{1}
=
\begin{pmatrix}
1
\\
1/6
\\
1/2
\\
1/4
\end{pmatrix}
u,
\qquad
u>0;
\qquad
\bm{x}_{2}
=
\begin{pmatrix}
1
\\
1/4
\\
1/2
\\
1/4
\end{pmatrix}
v,
\qquad
v>0.
\end{equation*}

Note that, in the constrained problem, the matrix $\bm{C}$ produces only one nontrivial inequality $x_{4}\leq x_{2}$, which means that the rating of alternative $2$ must be not less than that of $4$. Since the rating of alternative $2$ is always less than or equal to that of $4$ if no constraints are imposed, one can expect that the solution of the constrained problem will be the vector $\bm{x}_{2}$, at which the ratings of these alternatives become equal.  

We start with evaluation of the spectral radii $\lambda$ and $\mu$ given by \eqref{E-lambda-mu} for the matrices $\bm{A}$ and $\bm{B}$ of order $n=4$. By using the results obtained in \cite{Krivulin2020Using}, we can write
\begin{gather*}
\bm{A}^{2}
=
\begin{pmatrix}
1 & 8 & 4 & 16
\\
1/3 & 1 & 4/3 & 2
\\
2 & 12 & 1 & 4
\\
1 & 3 & 2 & 1
\end{pmatrix},
\qquad
\bm{A}^{3}
=
\begin{pmatrix}
8 & 48 & 4 & 16
\\
1 & 6 & 4/3 & 16/3
\\
4 & 12 & 8 & 4
\\
1 & 4 & 4 & 8
\end{pmatrix},
\\
\bm{A}^{4}
=
\begin{pmatrix}
16 & 48 & 32 & 16
\\
8/3 & 16 & 4 & 16/3
\\
4 & 16 & 16 & 32
\\
4 & 24 & 4 & 16
\end{pmatrix},
\qquad
\lambda
=
2.
\end{gather*}

Similar computations yield
\begin{gather*}
\bm{B}^{2}
=
\begin{pmatrix}
1 & 12 & 4 & 16
\\
1/2 & 1 & 2 & 4/3
\\
2 & 8 & 1 & 4
\\
1 & 2 & 2 & 1
\end{pmatrix},
\qquad
\bm{B}^{3}
=
\begin{pmatrix}
8 & 32 & 4 & 16
\\
2/3 & 6 & 2 & 8
\\
4 & 8 & 8 & 4
\\
1 & 6 & 4 & 8
\end{pmatrix},
\\
\bm{B}^{4}
=
\begin{pmatrix}
16 & 32 & 32 & 16
\\
4 & 16 & 8/3 & 8
\\
4 & 24 & 16 & 32
\\
4 & 16 & 4 & 16
\end{pmatrix},
\qquad
\mu
=
2.
\end{gather*}

Consider the matrix of constraints $\bm{C}$, and note that 
\begin{equation*}
\bm{C}^{2}
=
\bm{C}^{3}
=
\bm{C}^{4}
=
\bm{0},
\qquad
\mathop\mathrm{Tr}\bm{C}
=
0.
\end{equation*}

Since $\mathop\mathrm{Tr}\bm{C}<1$, we conclude that the assumptions of Theorem~\ref{T-minxAxxBx-Cxleqx} are fulfilled.

To take into account the constraints, we apply \eqref{E-sigma-theta} to evaluate the scalars $\sigma$ and $\theta$. By using properties of traces and eliminating terms that include powers of $\bm{C}$, we have
\begin{equation*}
\sigma
=
\mathop\mathrm{tr}(\bm{A}\bm{C})
\oplus
\mathop\mathrm{tr}\nolimits^{1/2}(\bm{A}^{2}\bm{C})
\oplus
\mathop\mathrm{tr}\nolimits^{1/2}((\bm{A}\bm{C})^{2})
\oplus
\mathop\mathrm{tr}\nolimits^{1/3}(\bm{A}^{3}\bm{C}).
\end{equation*}
  
We form the matrices and then take their traces to obtain 
\begin{gather*}
\bm{A}\bm{C}
=
\begin{pmatrix}
0 & 0 & 0 & 3
\\
0 & 0 & 0 & 1
\\
0 & 0 & 0 & 2
\\
0 & 0 & 0 & 3
\end{pmatrix},
\qquad
\bm{A}^{2}\bm{C}
=
\begin{pmatrix}
0 & 0 & 0 & 8
\\
0 & 0 & 0 & 1
\\
0 & 0 & 0 & 12
\\
0 & 0 & 0 & 3
\end{pmatrix},
\\
(\bm{A}\bm{C})^{2}
=
\begin{pmatrix}
0 & 0 & 0 & 9
\\
0 & 0 & 0 & 3
\\
0 & 0 & 0 & 6
\\
0 & 0 & 0 & 9
\end{pmatrix},
\qquad
\bm{A}^{3}\bm{C}
=
\begin{pmatrix}
0 & 0 & 0 & 48
\\
0 & 0 & 0 & 6
\\
0 & 0 & 0 & 12
\\
0 & 0 & 0 & 4
\end{pmatrix},
\qquad
\sigma
=
3.
\end{gather*}

In the same way, we calculate
\begin{gather*}
\bm{B}\bm{C}
=
\begin{pmatrix}
0 & 0 & 0 & 2
\\
0 & 0 & 0 & 1
\\
0 & 0 & 0 & 3
\\
0 & 0 & 0 & 2
\end{pmatrix},
\qquad
\bm{B}^{2}\bm{C}
=
\begin{pmatrix}
0 & 0 & 0 & 12
\\
0 & 0 & 0 & 1
\\
0 & 0 & 0 & 8
\\
0 & 0 & 0 & 2
\end{pmatrix},
\\
(\bm{B}\bm{C})^{2}
=
\begin{pmatrix}
0 & 0 & 0 & 4
\\
0 & 0 & 0 & 2
\\
0 & 0 & 0 & 6
\\
0 & 0 & 0 & 4
\end{pmatrix},
\qquad
\bm{B}^{3}\bm{C}
=
\begin{pmatrix}
0 & 0 & 0 & 32
\\
0 & 0 & 0 & 6
\\
0 & 0 & 0 & 8
\\
0 & 0 & 0 & 6
\end{pmatrix},
\qquad
\theta
=
2.
\end{gather*}

Furthermore, we check whether the Pareto frontier of the constrained problem degenerates to a single point. We use \eqref{E-Gs-Ht} with $n=4$ to construct the function
\begin{multline*}
H(t)
=
r_{3,3,1}
t^{-3}
\oplus
(r_{2,2,1}
\oplus
r_{3,2,1})
t^{-2}
\oplus
(r_{1,1,1}
\oplus
r_{2,1,1}
\oplus
r_{2,2,2}^{1/2}
\oplus
r_{3,1,1})
t^{-1}
\\\oplus
(r_{1,1,2}^{1/2}
\oplus
r_{2,1,2}^{1/2})
t^{-1/2}
\oplus
r_{1,1,3}^{1/3}
t^{-1/3}.
\end{multline*}

Application of \eqref{E-rklm} together with properties of traces yields
\begin{gather*}
r_{1,1,1}
=
\mathrm{tr}(\bm{A}\bm{B}),
\qquad
r_{1,1,2}
=
\mathop\mathrm{tr}(\bm{A}^{2}\bm{B}),
\qquad
r_{1,1,3}
=
\mathop\mathrm{tr}(\bm{A}^{3}\bm{B}),
\\
r_{2,1,1}
=
\mathop\mathrm{tr}(\bm{A}\bm{B}\bm{C})
\oplus
\mathop\mathrm{tr}(\bm{B}\bm{A}\bm{C}),
\qquad
r_{2,1,2}
=
\mathop\mathrm{tr}(\bm{A}^{2}\bm{B}\bm{C})
\oplus
\mathop\mathrm{tr}(\bm{B}\bm{A}^{2}\bm{C})
\oplus
\mathop\mathrm{tr}(\bm{A}\bm{B}\bm{A}\bm{C}),
\\
r_{2,2,1}
=
\mathop\mathrm{tr}(\bm{A}\bm{B}^{2}),
\qquad
r_{2,2,2}
=
\mathop\mathrm{tr}(\bm{A}^{2}\bm{B}^{2})
\oplus
\mathop\mathrm{tr}((\bm{A}\bm{B})^{2}),
\qquad
r_{3,1,1}
=
\mathop\mathrm{tr}(\bm{A}\bm{C}\bm{B}\bm{C}),
\\
r_{3,2,1}
=
\mathop\mathrm{tr}(\bm{A}\bm{B}^{2}\bm{C})
\oplus
\mathop\mathrm{tr}(\bm{B}\bm{A}\bm{B}\bm{C})
\oplus
\mathop\mathrm{tr}(\bm{B}^{2}\bm{A}\bm{C}),
\qquad
r_{3,3,1}
=
\mathop\mathrm{tr}(\bm{A}\bm{B}^{3}).
\end{gather*}

To evaluate traces, we first form the matrices 
\begin{gather*}
\bm{A}\bm{B}
=
\begin{pmatrix}
3/2 & 12 & 4 & 16
\\
1/2 & 3/2 & 4/3 & 2
\\
2 & 8 & 1 & 4
\\
3/2 & 3 & 2 & 3/2
\end{pmatrix},
\qquad
\bm{B}\bm{A}
=
\begin{pmatrix}
1 & 8 & 4 & 16
\\
1/2 & 3/2 & 2 & 4/3
\\
2 & 12 & 3/2 & 4
\\
2/3 & 3 & 2 & 1
\end{pmatrix},
\\
\bm{A}^{2}\bm{B}
=
\begin{pmatrix}
8 & 32 & 4 & 16
\\
1 & 4 & 4/3 & 16/3
\\
6 & 12 & 8 & 6
\\
3/2 & 6 & 4 & 8
\end{pmatrix},
\qquad
\bm{A}^{3}\bm{B}
=
\begin{pmatrix}
24 & 48 & 32 & 24
\\
3 & 32/3 & 4 & 16/3
\\
6 & 24 & 16 & 32
\\
4 & 16 & 4 & 16
\end{pmatrix},
\\
\bm{A}\bm{B}^{2}
=
\begin{pmatrix}
8 & 32 & 6 & 16
\\
1 & 4 & 2 & 16/3
\\
4 & 8 & 8 & 4
\\
3/2 & 6 & 6 & 8
\end{pmatrix},
\qquad
\bm{A}^{2}\bm{B}^{2}
=
\begin{pmatrix}
16 & 32 & 32 & 16
\\
8/3 & 32/3 & 4 & 16/3
\\
6 & 24 & 24 & 32
\\
4 & 16 & 6 & 16
\end{pmatrix},
\\
(\bm{A}\bm{B})^{2}
=
\begin{pmatrix}
24 & 48 & 32 & 24
\\
3 & 32/3 & 4 & 8
\\
6 & 24 & 32/3 & 32
\\
4 & 18 & 6 & 24
\end{pmatrix},
\qquad
\bm{A}\bm{B}^{3}
=
\begin{pmatrix}
16 & 32 & 32 & 24
\\
8/3 & 32/3 & 4 & 8
\\
4 & 24 & 16 & 32
\\
4 & 18 & 6 & 24
\end{pmatrix}.
\end{gather*}

Next, we apply these matrices to calculate
\begin{gather*}
\bm{A}\bm{B}\bm{C}
=
\begin{pmatrix}
0 & 0 & 0 & 12
\\
0 & 0 & 0 & 3/2
\\
0 & 0 & 0 & 8 &
\\
0 & 0 & 0 & 3
\end{pmatrix},
\qquad
\bm{B}\bm{A}\bm{C}
=
\begin{pmatrix}
0 & 0 & 0 & 8
\\
0 & 0 & 0 & 3/2
\\
0 & 0 & 0 & 12
\\
0 & 0 & 0 & 3
\end{pmatrix},
\\
\bm{A}^{2}\bm{B}\bm{C}
=
\begin{pmatrix}
0 & 0 & 0 & 32
\\
0 & 0 & 0 & 4
\\
0 & 0 & 0 & 12
\\
0 & 0 & 0 & 6
\end{pmatrix},
\qquad
\bm{B}\bm{A}^{2}\bm{C}
=
\begin{pmatrix}
0 & 0 & 0 & 48
\\
0 & 0 & 0 & 4
\\
0 & 0 & 0 & 12
\\
0 & 0 & 0 & 4
\end{pmatrix},
\\
\bm{A}\bm{B}\bm{A}\bm{C}
=
\begin{pmatrix}
0 & 0 & 0 & 48
\\
0 & 0 & 0 & 6
\\
0 & 0 & 0 & 12
\\
0 & 0 & 0 & 9/2
\end{pmatrix},
\qquad
\bm{A}\bm{B}^{2}\bm{C}
=
\begin{pmatrix}
0 & 0 & 0 & 32
\\
0 & 0 & 0 & 4
\\
0 & 0 & 0 & 8
\\
0 & 0 & 0 & 6
\end{pmatrix},
\\
\bm{B}\bm{A}\bm{B}\bm{C}
=
\begin{pmatrix}
0 & 0 & 0 & 32
\\
0 & 0 & 0 & 6
\\
0 & 0 & 0 & 12
\\
0 & 0 & 0 & 6
\end{pmatrix},
\qquad
\bm{B}^{2}\bm{A}\bm{C}
=
\begin{pmatrix}
0 & 0 & 0 & 48
\\
0 & 0 & 0 & 4
\\
0 & 0 & 0 & 12
\\
0 & 0 & 0 & 4
\end{pmatrix},
\\
\bm{A}\bm{C}\bm{B}\bm{C}
=
\begin{pmatrix}
0 & 0 & 0 & 6
\\
0 & 0 & 0 & 2
\\
0 & 0 & 0 & 4
\\
0 & 0 & 0 & 6
\end{pmatrix}.
\end{gather*}

Evaluating the traces of the matrices obtained results in
\begin{gather*}
r_{1,1,1}
=
3/2,
\qquad
r_{1,1,2}
=
8,
\qquad
r_{1,1,3}
=
24,
\qquad
r_{2,1,1}
=
3,
\qquad
r_{2,1,2}
=
6,
\\
r_{2,2,1}
=
8,
\qquad
r_{2,2,2}
=
24,
\qquad
r_{3,1,1}
=
6,
\qquad
r_{3,2,1}
=
6,
\qquad
r_{3,3,1}
=
24.
\end{gather*}

Finally, we construct the function
\begin{equation*}
H(t)
=
24t^{-3}
\oplus
8t^{-2}
\oplus
6t^{-1}
\oplus
8^{1/2}t^{-1/2}
\oplus
24^{1/3}t^{-1/3}.
\end{equation*}
 
We calculate $\lambda\oplus\sigma=3$, $\mu\oplus\theta=2$ and $H(\mu\oplus\theta)=H(2)=3$. Since the equality $H(\mu\oplus\theta)=\lambda\oplus\sigma$ is valid, it follows from Theorem~\ref{T-minxAxxBx-Cxleqx} that the Pareto frontier shrinks to the point
\begin{equation*}
\alpha
=
3,
\qquad
\beta
=
2,
\end{equation*}
whereas the solution is given by
\begin{equation*}
\bm{x}
=
(3^{-1}\bm{A}\oplus2^{-1}\bm{B}\oplus\bm{C})^{\ast}
\bm{u},
\qquad
\bm{u}
>
\bm{0}.
\end{equation*}

Furthermore, we consider the matrix
\begin{equation*}
3^{-1}\bm{A}\oplus2^{-1}\bm{B}\oplus\bm{C}
=
\begin{pmatrix}
1/2 & 1 & 2 & 1
\\
1/4 & 1/2 & 1/6 & 1
\\
1/8 & 3/2 & 1/2 & 2
\\
1/4 & 1 & 1/8 & 1/2
\end{pmatrix},
\end{equation*}
and calculate its second and third powers
\begin{equation*}
\begin{pmatrix}
1/4 & 3 & 1 & 4
\\
1/4 & 1 & 1/2 & 1/2
\\
1/2 & 2 & 1/4 & 3/2
\\
1/4 & 1/2 & 1/2 & 1
\end{pmatrix},
\qquad
\begin{pmatrix}
1 & 4 & 1/2 & 3
\\
1/4 & 3/4 & 1/2 & 1
\\
1/2 & 3/2 & 1 & 2
\\
1/4 & 1 & 1/2 & 1
\end{pmatrix}.
\end{equation*}

The Kleene star matrix, which generates the solutions, takes the form
\begin{equation*}
(3^{-1}\bm{A}\oplus2^{-1}\bm{B}\oplus\bm{C})^{\ast}
=
\begin{pmatrix}
1 & 4 & 2 & 4
\\
1/4 & 1 & 1/2 & 1
\\
1/2 & 2 & 1 & 2
\\
1/4 & 1 & 1/2 & 1
\end{pmatrix}.
\end{equation*}

Since all columns in this matrix are collinear, we take one of them, say the first, to write the solution as
\begin{equation*}
\bm{x}
=
\begin{pmatrix}
1
\\
1/4
\\
1/2
\\
1/4
\end{pmatrix}
u,
\qquad
u>0.
\end{equation*}

Specifically, with $u=1$, we have the vector of ratings $\bm{x}=(1,0.25,0.5,0.25)^{T}$.

Finally note that the obtained solution coincides with the vector $\bm{x}_{2}$ from the solution set of the unconstrained problem in \cite{Krivulin2020Using}, which is in agreement with the prior assessment.
\end{example}

To conclude this section, we observe that the computational scheme demonstrated by the example mainly involves simple algebraic manipulations with matrices and their traces, which can be directly extended to decision problems with a greater number of alternatives (with matrices of higher order). At the same time, the extension of the solution to problems with three and more criteria leads to more complicated analytical technique, which needs to be further developed.

\section{Conclusions}
\label{S-C}

In this paper, we have developed a novel application of tropical (idempotent) algebra to solve a new bi-criteria problem of rating alternatives through pairwise comparisons, subject to constraints on the relative values of ratings. A complete set of Pareto-optimal solutions of the problem has been obtained analytically in a compact vector form ready for further analysis and computations. This result extends previous solutions of an unconstrained bi-criteria problem and a constrained single-criterion problem.

The new solution provides decision-makers with a direct description of all Pareto-optimal decisions that satisfy additional conditions (constraints), which allows to improve the quality and efficiency of decision making in conditional problems in practice. Application examples include the problem of pairwise comparisons that are used to restore the total order on a partially ordered set of alternatives, and the two-stage procedure of evaluating alternatives from pairwise comparisons, where results of the first stage can override the results of the second.   

The proposed algebraic approach may serve to complement and supplement existing methods, including the heuristic numerical AHP method as well as well-justified analytical WGM method, which both find a single solution rather than derive all Pareto-optimal solutions of the multi-criteria decision-making problem under study. In addition, the results presented in the paper, demonstrate a strong potential of the approach to offer complete analytical solutions of pairwise comparison problems with various constraints, which are difficult to handle by existing techniques and are still not adequately addressed in the literature.   

As a limitation of the approach, one can consider the complexity of the algebraic expressions in the analytical solution, which rapidly growths as the number of criteria increases. However, when the number of criteria is small (say, less than ten), we expect that appropriate application of computer algebra systems may help to overcome this difficulty.
 
We suggest future research that focuses on extending the results to bi-criteria problems with additional constraints and to problems with more than two criteria. The evaluation of computational complexity of the solution is also of interest.

\bibliographystyle{abbrvurl}

\bibliography{Algebraic_solution_to_constrained_bi-criteria_decision_problem_of_rating_alternatives_through_pairwise_comparisons}

\end{document}